\documentclass[12pt,a4paper]{article}
\usepackage{latexsym,amssymb,amsfonts,amsmath,amsthm,nccmath,float,enumitem,setspace,xcolor,bm}
\usepackage[hidelinks]{hyperref}
\usepackage[margin=2cm]{geometry}

\hypersetup{
    colorlinks,
    linkcolor={red!80!black},
    citecolor={blue!80!black},
    urlcolor={blue!80!black}
}

\newtheorem{theorem}{Theorem}
\newtheorem{corollary}[theorem]{Corollary}
\newtheorem{lemma}[theorem]{Lemma}

\theoremstyle{definition}

\def \leq {\leqslant}
\def \geq {\geqslant}
\def \Z {\mathbb{Z}}

\def \B {\mathcal{B}}
\def \L {\mathnormal{L}}
\def \R {\mathnormal{R}}
\def \l {\lambda}

\let\oldproofname=\proofname
\renewcommand{\proofname}{\rm\bf{\oldproofname}}

\title{Block avoiding point sequencings of partial Steiner systems}

\author{
\textsc{Daniel Horsley}
\               \thanks{\textit{E-mail: danhorsley@gmail.com}}\\
\textit{\footnotesize{School of Mathematics,}}\\
\textit{\footnotesize{Monash University, VIC 3800, Australia.}}\\
\textsc{Padraig \'O Cath\'ain}
				\thanks{\textit{E-mail: p.ocathain@gmail.com}}\\
\textit{\footnotesize{Fiontar \& Scoil na Gaeilge,}}\\
\textit{\footnotesize{Dublin City University, Collins Ave Ext, Whitehall, Dublin 9, Ireland.}}\\}

\date{}

\begin{document}
\maketitle
\setstretch{1.1}

\begin{abstract}
A \emph{partial $(n,k,t)_\l$-system} is a pair $(X,\B)$ where $X$ is an $n$-set of \emph{vertices} and $\B$ is a collection of $k$-subsets of $X$ called \emph{blocks} such that each $t$-set of vertices is a subset of at most $\l$ blocks. A \emph{sequencing} of such a system is a labelling of its vertices with distinct elements of $\{0,\ldots,n-1\}$. A sequencing is \emph{$\ell$-block avoiding} or, more briefly, \emph{$\ell$-good} if no block is contained in a set of $\ell$ vertices with consecutive labels. Here we give a short proof that, for fixed $k$, $t$ and $\l$, any partial $(n,k,t)_\l$-system has an $\ell$-good sequencing for some $\ell=\Theta(n^{1/t})$ as $n$ becomes large. This improves on results of Blackburn and Etzion, and of Stinson and Veitch. Our result is perhaps of most interest in the case $k=t+1$ where results of Kostochka, Mubayi and Verstra\"{e}te show that the value of $\ell$ cannot be increased beyond $\Theta((n \log n)^{1/t})$. A special case of our result shows that every partial Steiner triple system (partial $(n,3,2)_1$-system) has an $\ell$-good sequencing for each positive integer $\ell \leq 0.0908\,n^{1/2}$.
\end{abstract}

\section{Introduction}
For positive integers $n$, $k$, $t$ and $\l$ with $n \geq k > t \geq 2$, a \emph{partial $(n,k,t)_\l$-system} is a pair $(X,\B)$ where $X$ is an $n$-set of \emph{vertices} and $\B$ is a collection of $k$-subsets of $X$ called \emph{blocks} such that each $t$-set of vertices is a subset of at most $\l$ blocks. Such a partial system is a (complete) \emph{$(n,k,t)_\l$-system} if each $t$-set of vertices is a subset of exactly $\l$ blocks. Note that $(n,k,t)_\l$-systems are often referred to as $t$-$(n,k,\l)$ designs elsewhere. An $(n,k,t)_1$-system is referred to as an \emph{$(n,k,t)$-Steiner system}.  Historically, the case where $k=t+1$ has attracted particular interest. An $(n,3,2)$-Steiner system is often called a \emph{Steiner triple system} and an $(n,4,3)$-Steiner system is often called a \emph{Steiner quadruple system}. An \emph{independent set} of a partial $(n,k,t)_\l$-system $(X,\B)$ is a subset $Y$ of $X$ such that no block in $\B$ is a subset of $Y$, and the \emph{independence number} of $(X,\B)$ is the cardinality of its largest independent set.

For each positive integer $n$, we let $\Z_n$ denote the additive cyclic group with elements $\{0,\ldots,n-1\}$. Let $(X,\mathcal{B})$ be a partial $(n,k,t)_\l$-system. A \emph{sequencing} of $(X,\mathcal{B})$ is a bijection $\varphi: \Z_n \rightarrow X$, which we represent as an $n$-tuple $(\varphi(0),\ldots,\varphi(n-1))$. We say a subset $S$ of $X$ is \emph{consecutive} under this sequencing if $S=\{\varphi(i),\ldots,\varphi(i+|S|-1)\}$ for some integer $i \in \{0,\ldots,n-|S|\}$, and we say that $S$ is \emph{cyclically consecutive} under this sequencing if $S=\{\varphi(i),\ldots,\varphi(i+|S|-1)\}$ for some $i \in \Z_n$ where the addition takes place in $\Z_n$. For a positive integer $\ell$, such a sequencing is \emph{$\ell$-block avoiding} or, more briefly, \emph{$\ell$-good} if each set of $\ell$ consecutive vertices is an independent set of $(X,\mathcal{B})$. It is \emph{cyclically $\ell$-good} if each set of $\ell$ cyclically consecutive vertices is an independent set of $(X,\mathcal{B})$. These definitions accord with those given in \cite{BlaEtz,KreSti}, where good and cyclically good sequencings were studied. Trivially, any cyclically $\ell$-good sequencing of $(X,\mathcal{B})$ is cyclically $\ell'$-good and $\ell'$-good for each $\ell' \in \{1,\ldots,\ell\}$, and any sequencing of $(X,\mathcal{B})$ is cyclically $(k-1)$-good.

It is worth observing that the existence of an $\ell$-good sequencing of a partial $(n,k,t)_\l$-system carries a number of implications about well-studied parameters of the system. Specifically, a partial $(n,k,t)_\l$-system $(X,\B)$ with an $\ell$-good sequencing:
\begin{itemize}
    \item
has independence number at least $\ell$ because any set of $\ell$ consecutive vertices in the sequencing is an independent set;
    \item
has chromatic number\footnote{The \emph{chromatic number} of $(X,\B)$ is the smallest number of independent sets into which $X$ can be partitioned.} at most $\lceil\frac{n}{\ell}\rceil$ because we can partition $X$ into $\lceil\frac{n}{\ell}\rceil$ sets, each of at most $\ell$ consecutive vertices and hence independent;
    \item
has fractional chromatic number\footnote{The \emph{fractional chromatic number} of $(X,\B)$ is the minimum total weight assigned by any assignment of nonnegative weights to the independent sets of $(X,\B)$ such that, for each $x \in X$, the sum of the weights of the independent sets containing $x$ is at least 1.} at most $\frac{n}{\ell}$ if the sequencing is cyclically $\ell$-good, because we can assign weight $\frac{1}{\ell}$ to each set of $\ell$ cyclically consecutive vertices and weight 0 to every other independent set.
\end{itemize}

In this paper, we are concerned with the question of determining, for a given parameter set $(n,k,t)_\l$, the greatest value of $\ell$ such that each partial $(n,k,t)_\l$-system has a cyclically $\ell$-good sequencing. In particular, we are concerned with this problem for large values of $n$, and all the asymptotic notation we employ is relative to a regime in which $n$ is large and $k$, $t$ and $\lambda$ are fixed. It was proved in \cite{KreSti} that every partial Steiner triple system of order at least 4 has a 3-good sequencing and that every partial Steiner triple system of order at least 72 has a 4-good sequencing. Stinson and Veitch \cite{StiVei} showed that any Steiner triple system of order $n$ has an $\ell$-good sequencing for each positive integer $\ell \leq (16n)^{1/6}$. Blackburn and Etzion \cite{BlaEtz} improved this bound to $\ell \leq (1+o(1))(2n)^{1/4}$. They also showed that any Steiner triple system of order $n$ has a cyclically $\ell$-good sequencing for each positive integer $\ell \leq (1+o(1))(\frac{2}{3}n)^{1/4}$ and that any $(n,k,t)_\l$-system has a cyclically $\ell$-good sequencing for some integer $\ell=\Theta(n^{1/2t})$. We direct the reader to \cite{BlaEtz,StiVei} for discussion of other interesting questions about $\ell$-good sequencings, but these do not concern us here. Results on $\ell$-good sequencings of directed and Mendelsohn triple systems can be found in \cite{KreStiVei1,KreStiVei2}. See \cite{ErsGri} for a recent preprint showing that every Steiner triple system of order at least 13 has a 4-good sequencing and that every $3$-chromatic Steiner triple system of order at least 15 has a 5-good sequencing.

Here we present a short proof establishing that any $(n,k,t)_\l$-system has a cyclically $\ell$-good sequencing for some integer $\ell=\Theta(n^{1/t})$. Unlike the proofs in \cite{BlaEtz,StiVei}, which provide deterministic algorithms, our proof relies on an application of the Lov\'{a}sz local lemma. We state our main result below in Theorem~\ref{T:main}. Throughout the paper, $e$ denotes the base of the natural logarithm.

\begin{theorem}\label{T:main}
Let $k$, $t$ and $\l$ be positive integers such that $k \geq t+1$ and $t\geq2$, and let $\alpha$ be the unique positive real number satisfying
\begin{equation}\label{E:k=t+1}
\mfrac{\l(2^{t+1}-1)}{t!}\,\alpha^t+2\left(\mfrac{e\l(t+1)(2^{t+1}-1)}{t!}\right)^{\frac{1}{t}}\alpha=1.
\end{equation}
Every partial $(n,k,t)_\l$-system has a cyclically $\ell$-good sequencing for each positive integer $\ell \leq \alpha n^{1/t}$.
\end{theorem}

As we discuss further below, we believe this result is of most interest in the case $k=t+1$. However, we note that we can improve the constant in the case where $k \geq t+2$ and $n$ is large.

\begin{theorem}\label{T:main2}
Let $k$, $t$ and $\l$ be fixed positive integers such that $k \geq t+2$ and $t\geq2$, and let $\alpha$ be a fixed positive constant such that
\begin{equation}\label{E:k>=t+2}
\alpha<\left(\mfrac{(k-1)\cdots(k-t)}{\l(2^{t+1}-1)}\right)^{\frac{1}{t}}.
\end{equation}
For all sufficiently large $n$, every partial $(n,k,t)_\l$-system has a cyclically $\ell$-good sequencing for each positive integer $\ell \leq \alpha n^{1/t}$.
\end{theorem}

Kostochka, Mubayi and Verstra\"{e}te \cite{KosMubVer} have shown that there are partial $(n,t+1,t)_\l$-systems with independence number $\Theta((n \log n)^{1/t})$. Since any consecutive set of $\ell$ vertices in an $\ell$-good sequencing is independent, this implies that $\ell=\Theta(n^{1/t})$ in Theorem~\ref{T:main} cannot be improved beyond $\ell=\Theta((n \log n)^{1/t})$, even if we only demand that the sequencing be $\ell$-good rather than cyclically $\ell$-good. In particular, it cannot be improved to $\ell=\Theta(n^{1/t+\epsilon})$ for any $\epsilon > 0$. On the other hand, it has been shown that any partial $(n,k,t)$-Steiner system has independence number at least $\Theta((n^{k-t}\log n)^{1/(k-1)})$ \cite{RodSiv} and chromatic number at most $\Theta((n^{t-1}/\log n)^{1/(k-1)})$ \cite[Theorem 10]{KosMubRodTet}. Given that Theorem~\ref{T:main} with $k=t+1$ achieves a value of $\ell$ within a logarithmic factor of the minimum independence number, the results of \cite{KosMubRodTet,RodSiv} lead us to suspect that our bounds for $k \geq t+2$ are far from optimal, especially when $k$ is large in comparison to $t$.

We give the consequences of Theorem~\ref{T:main} in the case of partial Steiner triple and quadruple systems. In each case we round the appropriate value of $\alpha$ down to three significant figures.

\begin{corollary}\label{C:STS}
Every partial Steiner triple system of order $n$ has a cyclically $\ell$-good sequencing for each positive integer $\ell \leq 0.0908\,n^{1/2}$.
\end{corollary}

\begin{corollary}\label{C:SQS}
Every partial Steiner quadruple system of order $n$ has a cyclically $\ell$-good sequencing for each positive integer $\ell \leq 0.164\,n^{1/3}$.
\end{corollary}

Extending the aforementioned result of \cite{RodSiv}, Eustis and Verstra\"{e}te \cite{EusVer} have shown that there are infinitely many $n$ for which there is a Steiner triple system of order $n$ which has independence number $(1+o(1))(3n\log n)^{1/2}$. Consequently, the bound in Corollary~\ref{C:STS} cannot be improved beyond this, even if we demand that the system be a (complete) Steiner triple system and only demand that the sequencing be $\ell$-good rather than cyclically $\ell$-good.

\section{Proof of Theorems~\ref{T:main} and \ref{T:main2}}

Our proof is based on a concept we define below as an $\ell$-presequencing. This is, in turn, defined in terms of objects we call $\ell$-buffered sets. For a positive integer $\ell$, an \emph{$\ell$-buffered set} is a 3-tuple $(S,S^{\L},S^{\R})$ where $S$ is a set and $S^{\L}$ and $S^{\R}$ are subsets of $S$, called \emph{buffers}, such that
\begin{itemize}
    \item[(B1)]
if $|S| \leq \ell-2$, then $S^{\L}=S^{\R}=S$;
    \item[(B2)]
if $\ell-1 \leq |S| \leq 2\ell-2$, then $|S^{\L}|=|S^{\R}|=\ell-1$ and $S^{\L} \cup S^{\R} =S$; and
    \item[(B3)]
if $|S| \geq 2\ell-1$, then $|S^{\L}|=|S^{\R}|=\ell-1$ and $S^{\L} \cap S^{\R} = \emptyset$.
\end{itemize}
We usually write such a buffered set simply as $S$ with the assumption that its two buffers are denoted $S^{\L}$ and $S^{\R}$. Such a buffered set is \emph{deficient} if $|S| \leq \ell-2$ and \emph{nondeficient} otherwise. An \emph{$\ell$-presequencing} of $(n,k,t)_\l$-system $(X,\B)$ is a tuple $(X_0,\ldots,X_{s-1})$ of $\ell$-buffered sets such that $\{X_0,\ldots,X_{s-1}\}$ is a partition of $X$ and the following hold.
\begin{itemize}
    \item[(P1)]
For each $i \in \Z_s$, $X_i$ is an independent set in $(X,\B)$.
    \item[(P2)]
For each $i \in \Z_s$, $X^{\R}_i \cup X^{\L}_{i+1}$ is an independent set in $(X,\B)$.
\end{itemize}
We call the buffered sets $X_0,\ldots,X_{s-1}$ the \emph{classes} of the $\ell$-presequencing and we say an $\ell$-presequencing is \emph{nondeficient} if each of its classes is nondeficient. As we observe in the following lemma, the existence of a nondeficient $\ell$-presequencing implies the existence of a cyclically $\ell$-good sequencing.

\begin{lemma}\label{L:nondefToSeq}
If a partial $(n,k,t)_\l$-system has a nondeficient $\ell$-presequencing then it has a cyclically $\ell$-good sequencing.
\end{lemma}

\begin{proof}
Let $(X,\B)$ be a partial $(n,k,t)_\l$-system and suppose that $(W_0,\ldots,W_{s-1})$ is a nondeficient $\ell$-presequencing of $(X,\B)$. For each $i \in \Z_s$, let $w_i=|W_i|$ and note that $w_i \geq \ell-1$ since $(W_0,\ldots,W_{s-1})$ is nondeficient. By (B2) and (B3), this means that $|W^{\L}_i|=|W^{\R}_i|=\ell-1$ for each $i \in \Z_s$. Roughly speaking, we form a sequencing of $(X,\B)$ by concatenating sequencings of $W_0,\ldots,W_{s-1}$ where, for each $i \in \Z_s$, the sequencing of $W_i$ begins with the elements of $W^{\L}_i$ and ends with the elements of $W^{\R}_i$. More formally, we can form a sequencing
\[\bigl(x_{0,0},x_{0,1},\ldots,x_{0,w_0-1},x_{1,0},x_{1,1},\ldots,x_{1,w_1-1},\ldots\ldots\ldots,x_{s-1,0},x_{s-1,1},\ldots,x_{s-1,w_{s-1}-1}\bigr),\]
of $(X,\B)$ such that, for each $i \in \Z_s$,
\[\{x_{i,0}, \ldots, x_{i,w_i-1}\}=W_i, \quad \{x_{i,0}, \ldots, x_{i,\ell-2}\}=W^{\L}_i \quad\text{and}\quad \{x_{i,w_i-\ell+1},\ldots,x_{i,w_i-1}\}=W^{\R}_i.\]
Any cyclically consecutive set $C$ of $\ell$ vertices in this sequencing satisfies either $C \subseteq W_i$ or $C \subseteq W^{\R}_i \cup W^{\L}_{i+1}$ for some $i \in \Z_s$. Thus, because $(W_0,\ldots,W_{s-1})$ obeys (P1) and (P2), this sequencing is $\ell$-good.
\end{proof}

Next we show that an $\ell$-presequencing with $s$ classes of a partial $(n,k,t)_\l$-system can be modified to create a nondeficient $\ell$-presequencing with $s$ classes, provided that $n$ is large enough as a function of $\l$, $k$, $t$, $\ell$ and $s$. The modification proceeds by repeatedly moving a carefully selected vertex from a nondeficient class into a deficient one.

\begin{lemma}\label{L:defToNondef}
If a partial $(n,k,t)_\l$-system has an $\ell$-presequencing with $s$ classes,
then it has a nondeficient $\ell$-presequencing with $s$ classes provided that
\begin{equation}\label{E:defToNondefCondition}
n > \l\Bigl(2\tbinom{2\ell-3}{t}-\tbinom{\ell-2}{t}\Bigr)\tbinom{k-1}{t}^{-1}+(2s-1)(\ell-1)-1.
\end{equation}
\end{lemma}

\begin{proof}
Let $(X,\B)$ be a partial $(n,k,t)_\l$-system and suppose that $(U_0,\ldots,U_{s-1})$ is an $\ell$-presequencing of $(X,\B)$. We will show that if $|U_j| \leq \ell-2$ for some $j \in \Z_s$, then there is an $\ell$-presequencing $(V_0,\ldots,V_{s-1})$ of $(X,\B)$ such that, for some $k \in \Z_s$ with $|V_k| \geq \ell$ and for some $y \in V_k$, we have $V_j=U_j \cup \{y\}$, $V_k=U_k \setminus \{y\}$ and $V_i=U_i$ for all $i \in \Z_s \setminus \{j,k\}$. This will suffice to complete the proof because iteratively applying this procedure will eventually result in a nondeficient $\ell$-presequencing.

Suppose that $|U_j| \leq \ell-2$ for some $j \in \Z_s$. Let $X^\dag$ be the set of vertices in the buffers, that is,
\[X^\dag=\medop\bigcup_{i=0}^{s-1}  (U^{\L}_i \cup  U^{\R}_i).\]
From (B1), $U^{\L}_j = U^{\R}_j =U_j$ and hence $|U^{\L}_j \cup U^{\R}_j|=|U_j| \leq \ell-2$. From (B1)--(B3), for each $i \in \Z_s \setminus \{j\}$, we have $|U^{\L}_{i} \cup U^{\R}_{i}| \leq 2\ell-2$. Thus
\begin{equation}\label{E:XdagSize}
|X^\dag| \leq (2s-1)(\ell-1)-1.
\end{equation}
Also note that $U^{\R}_{j-1} \cup U_j \cup U^{\L}_{j+1} \subseteq X^\dag$ since $U^{\L}_j = U^{\R}_j =U_j$. Now, we let $X^\ddag$ be the set of all vertices $x$ not in the buffers for which moving $x$ into $U_j$ would result in a block becoming a subset of $U^{\R}_{j-1} \cup U_j$ or of $U_j\cup U^{\L}_{j+1}$. That is, let
\[X^\ddag=\left\{x \in X \setminus X^\dag: \mbox{$B \setminus \{x\} \subseteq U^{\R}_{j-1} \cup U_j$ or $B \setminus \{x\} \subseteq  U_j \cup U^{\L}_{j+1}$ for some $B \in \B$}\right\}.\]
Observe that $|U^{\R}_{j-1}|, |U^{\L}_{j+1}| \leq \ell-1$ by (B1)--(B3). So the number of $t$-sets of vertices that are subsets of $U^{\R}_{j-1} \cup U_j$ or of $U_j\cup U^{\L}_{j+1}$ is at most
$2\tbinom{\ell-1+|U_j|}{t}-\tbinom{|U_j|}{t} \leq 2\tbinom{2\ell-3}{t}-\tbinom{\ell-2}{t}$,
where the inequality follows because $|U_j| \leq \ell-2$ by assumption. Each block of $\B$ containing $k-1$ elements of $U^{\R}_{j-1} \cup U_j$ or $k-1$ elements of $U_j\cup U^{\L}_{j+1}$ is a superset of $\binom{k-1}{t}$ of these $t$-subsets. Thus \begin{equation}\label{E:XddagSize}
|X^\ddag| \leq \l\Bigl(2\tbinom{2\ell-3}{t}-\tbinom{\ell-2}{t}\Bigr)\tbinom{k-1}{t}^{-1}.
\end{equation}

By \eqref{E:XdagSize}, \eqref{E:XddagSize} and the hypothesised lower bound on $n$, we have $n > |X^\dag| + |X^\ddag|$ and hence there is a vertex $y \in X \setminus (X^\dag \cup X^\ddag)$. Then $y \in U_k \setminus (U^{\L}_k \cup U^{\R}_k \cup X^\ddag)$ for some $k \in \Z_s \setminus \{j\}$. This implies that $|U_k| \geq 2\ell-1 \geq \ell$ by (B1) and (B2). We claim that $(V_0,\ldots,V_{s-1})$ is an $\ell$-presequencing of $(X,\B)$ where
\begin{itemize}
    \item
$(V_j,V^{\L}_j,V^{\R}_j)=(U_j \cup \{y\},U_j \cup \{y\},U_j \cup \{y\})$;
     \item
$(V_k,V^{\L}_k,V^{\R}_k)=(U_k \setminus \{y\},U^{\L}_k,U^{\R}_k)$; and
    \item
$(V_i,V^{\L}_i,V^{\R}_i)=(U_i,U^{\L}_i,U^{\R}_i)$ for all $i \in \Z_s \setminus \{j,k\}$.
\end{itemize}
To see this, observe that each class of $(V_0,\ldots,V_{s-1})$ is indeed an $\ell$-buffered set because $|U_j| \leq \ell-1$, $y \in U_k \setminus (U^{\L}_k \cup U^{\R}_k)$, and each class of $(U_0,\ldots,U_{s-1})$ is an $\ell$-buffered set. Further, (P1) and (P2) hold for $(V_0,\ldots,V_{s-1})$ because they held for $(U_0,\ldots,U_{s-1})$ and $y \notin X^\ddag$. So $(V_0,\ldots,V_{s-1})$ is an $\ell$-presequencing of $(X,\B)$ with the properties we desired and the proof is complete.
\end{proof}

In view of Lemmas~\ref{L:nondefToSeq} and \ref{L:defToNondef}, our final step is to show that partial $(n,k,t)_\l$-systems admit $\ell$-presequencings. This can be accomplished using the symmetric version of the Lov\'{a}sz local lemma, which we state below.

\begin{lemma}[see {\cite[Corollary 5.1.2]{AloSpe}}]\label{L:LLL}
Let $E_1,\ldots,E_r$ be events in a probability space such that each of these events has probability at most $p$ and is mutually independent from all but at most $d$ of the others. If $ep(d+1) \leq 1$, then there is a positive probability that no event in $\{E_1,\ldots,E_r\}$ occurs.
\end{lemma}

\begin{lemma}\label{L:findDef}
Any partial $(n,k,t)_\l$-system has an $\ell$-presequencing $(U_0,\ldots,U_{s-1})$ such that $s=\lceil\sigma\rceil$, where
\[\sigma=\sigma(n,k,t,\l)=\Bigl(e(2^k-1)\left(k\l\tbinom{n-1}{t-1}\tbinom{k-1}{t-1}^{-1}-k+1\right)\Bigr)^{\frac{1}{k-1}}.\]
\end{lemma}

\begin{proof}
Let $(X,\B)$ be a partial $(n,k,t)_\l$-system and let $s=\lceil\sigma\rceil$. Suppose for the moment that there is an indexed partition $\{U_0,\ldots,U_{s-1}\}$ of $X$ such that $U_i \cup U_{i+1}$ is an independent set of $(X,\B)$ for each $i \in \Z_s$. For each $i \in \Z_s$, we can easily choose subsets $U^{\L}_i$ and $U^{\R}_i$ of $U_i$ that obey (B1)--(B3). The resulting tuple $(U_0,\ldots,U_{s-1})$ of $\ell$-buffered sets certainly obeys (P1) and (P2) and hence is an $\ell$-presequencing of $(X,\B)$. Thus it suffices to prove there is such an indexed partition $\{U_0,\ldots,U_{s-1}\}$. We do so using the Lov\'{a}sz local lemma.

Consider forming an indexed partition $\{U_0,\ldots,U_{s-1}\}$ of $X$ by assigning each vertex in $X$ to one of the partition classes independently uniformly at random. For each block $B \in \B$, let $E_B$ be the event that $B \subseteq U_i \cup U_{i+1}$ for some $i \in \Z_s$. Let $B$ be a block in $\B$ and $j$ be an element of $\Z_s$. Then the probability that $B \subseteq U_j \cup U_{j+1}$ is $(\frac{2}{s})^k$, the probability that $B \subseteq U_{j+1}$ is $(\frac{1}{s})^k$ and hence the probability that $B \subseteq U_j \cup U_{j+1}$ but $B \nsubseteq U_{j+1}$ is $(\frac{2}{s})^k-(\frac{1}{s})^k$. Taking the sum over $j \in \Z_s$, we see that the probability $p$ of $E_B$ is given by
\begin{equation}\label{E:findDefp}
p=s\left(\mfrac{2^k}{s^k}-\mfrac{1}{s^k}\right)=\mfrac{2^k-1}{s^{k-1}}.
\end{equation}

For each $B \in \B$ and $\B' \subseteq \B$, $E_B$ is mutually independent from the set of events $\{E_{B'}:B' \in \B'\}$ provided that each block in $\B'$ is disjoint from $B$. Let $x$ be a vertex in $B$. The number of $t$-subsets of $X$ that contain $x$ is $\binom{n-1}{t-1}$. Each block in $\B$ that contains $x$ (including $B$ itself) is a superset of $\binom{k-1}{t-1}$ of these $t$-subsets. Thus there are at most $\l\binom{n-1}{t-1}/\binom{k-1}{t-1}-1$ blocks in $\B \setminus \{B\}$ that contain $x$. Taking the sum over $x \in B$, we see that the number $d$ of blocks in $\B \setminus \{B\}$ that contain at least one vertex in $B$ obeys
\begin{equation}\label{E:findDefd}
d \leq k\l\tbinom{n-1}{t-1}\tbinom{k-1}{t-1}^{-1}-k.
\end{equation}
Thus, $E_B$ is mutually independent from all but at most $d$ of the events in $\{E_{B}:B \in \B \setminus \{B\}\}$.

Now, using \eqref{E:findDefp} and \eqref{E:findDefd} and the fact that $s \geq \sigma$, we have that $ep(d+1) \leq 1$. Thus, by Lemma~\ref{L:LLL}, there is a positive probability that no event in $\{E_B:B \in \B\}$ occurs. This implies there is an indexed partition $\{U_0,\ldots,U_{s-1}\}$ of $X$ such that $U_i \cup U_{i+1}$ is an independent set of $(X,\B)$ for each $i \in \Z_s$ and the proof is complete.
\end{proof}

From Lemmas~\ref{L:nondefToSeq}, \ref{L:defToNondef} and \ref{L:findDef} it follows that, for any parameter set $(n,k,t)_\l$, every partial $(n,k,t)_\l$-system has a cyclically $\ell$-good sequencing provided that \eqref{E:defToNondefCondition} holds with $s=\lceil\sigma\rceil$ where $\sigma$ is as defined in Lemma~\ref{L:findDef}. We use this fact to prove Theorems~\ref{T:main} and \ref{T:main2}.

\begin{proof}[\textbf{\textup{Proof of Theorem~\ref{T:main}.}}]
Given a partial $(n,k,t)_\l$-system with $k \geq t+2$, one can delete $k-t-1$ vertices from each of its blocks to obtain a partial $(n,t+1,t)_\l$-system. Any cyclically $\ell$-good sequencing of the resulting partial system is also a cyclically $\ell$-good sequencing of the original partial system. Thus it suffices to prove the theorem in the case where $k=t+1$.

Let $(X,\B)$ be a partial $(n,t+1,t)_\l$-system, let $k=t+1$, and let $s=\lceil\sigma\rceil$ where $\sigma$ is as defined in Lemma~\ref{L:findDef}. Let $\alpha$ be the positive real number that obeys \eqref{E:k=t+1} and let $\ell = \lfloor\alpha n^{1/t}\rfloor$. We may assume that $\ell > t$ for otherwise the result is trivial. As discussed above, there is an $\ell$-good sequencing of $(X,\B)$ provided that \eqref{E:defToNondefCondition} holds. Using $k=t+1$, \eqref{E:defToNondefCondition} simplifies to
\begin{equation}\label{E:defToNondefConditionSpec}
n > \l\Bigl(2\tbinom{2\ell-3}{t}-\tbinom{\ell-2}{t}\Bigr)+(2s-1)(\ell-1)-1.
\end{equation}
Now,
\begin{equation}\label{E:proofUpper1}
(2s-1)(\ell-1) < 2\ell\sigma + \ell < 2\ell\left(\mfrac{e\l(t+1)(2^{t+1}-1)}{t}\tbinom{n-1}{t-1}\right)^{\frac{1}{t}} +\ell \leq 2\alpha\Bigl(\mfrac{e\l(t+1)(2^{t+1}-1)}{t!}\Bigr)^{\frac{1}{t}}n + \ell
\end{equation}
where the first inequality follows using $s<\sigma+1$, the second follows using the definition of $\sigma$ and $k=t+1$, and the third follows using $\binom{n-1}{t-1} < \frac{n^{t-1}}{(t-1)!}$ and $\ell \leq \alpha n^{1/t}$. Also,
\begin{multline}
  \l\Bigl(2\tbinom{2\ell-3}{t}-\tbinom{\ell-2}{t}\Bigr) \leq \mfrac{\l(\ell-2)}{t} \Bigl(4\tbinom{2\ell-3}{t-1}-\tbinom{\ell-3}{t-1}\Bigr) < \mfrac{\l(\ell-2)}{t} \left(\mfrac{4(2\ell)^{t-1}}{(t-1)!}-\mfrac{\ell^{t-1}}{(t-1)!}\right)= \\[1mm]
\mfrac{\l(2^{t+1}-1)(\ell-2)\ell^{t-1}}{t!} < \mfrac{\l(2^{t+1}-1)\ell^t}{t!} - \ell \leq  \mfrac{\l(2^{t+1}-1)\alpha^t}{t!}\,n - \ell. \label{E:proofUpper2}
\end{multline}
The second inequality above follows because $\frac{4(2\ell)^{t-1}}{(t-1)!}-4\binom{2\ell-3}{t-1} > \frac{\ell^{t-1}}{(t-1)!}-\binom{\ell-3}{t-1}$ since $f(a)=\frac{a^{t-1}}{(t-1)!}-\binom{a-3}{t-1}$ is an increasing function of $a$. The third follows by observing that $2\l(2^{t+1}-1)\ell^{t-1}>t!\ell$ since $\ell > t$. Using \eqref{E:proofUpper1} and \eqref{E:proofUpper2}, the right hand side of \eqref{E:defToNondefConditionSpec} is strictly less than
\[\left(\mfrac{\l(2^{t+1}-1)}{t!}\alpha^t + 2\Bigl(\mfrac{e\l(t+1)(2^{t+1}-1)}{t!}\Bigr)^{\frac{1}{t}}\alpha\right)n\]
and so \eqref{E:defToNondefConditionSpec} holds because $\alpha$ obeys \eqref{E:k=t+1}.
\end{proof}


\begin{proof}[\textbf{\textup{Proof of Theorem~\ref{T:main2}.}}]
Let $(X,\B)$ be a partial $(n,k,t)_\l$-system and let $s=\lceil\sigma\rceil$ where $\sigma$ is as defined in Lemma~\ref{L:findDef}. Let $\alpha$ be a positive constant that obeys \eqref{E:k>=t+2} and let $\ell = \lfloor\alpha n^{1/t}\rfloor$. Again, there is an $\ell$-good sequencing of $(X,\B)$ provided that \eqref{E:defToNondefCondition}, which we restate below, holds.
\begin{equation*}
n > \l\Bigl(2\tbinom{2\ell-3}{t}-\tbinom{\ell-2}{t}\Bigr)\tbinom{k-1}{t}^{-1}+(2s-1)(\ell-1)-1 \tag{\ref{E:defToNondefCondition}}
\end{equation*}
Observe that
\begin{equation}\label{E:proofAsymp1}
  \l\Bigl(2\tbinom{2\ell-3}{t}-\tbinom{\ell-2}{t}\Bigr)\tbinom{k-1}{t}^{-1} < \l\left(\mfrac{2^{t+1}\ell^t}{t!}-\mfrac{\ell^t}{t!}\right)\tbinom{k-1}{t}^{-1}
  \leq \mfrac{\l(2^{t+1}-1)\alpha^t}{(k-1)\cdots(k-t)}\,n
\end{equation}
where the first inequality follows because $\frac{2(2\ell)^{t}}{t!}-2\binom{2\ell-3}{t} > \frac{\ell^{t}}{t!}-\binom{\ell-2}{t}$ and the second follows using $\ell^t \leq \alpha^t n$. Furthermore $\ell=\Theta(n^{1/t})$ and, using the definition of $\sigma$, $s=\Theta(n^{(t-1)/(k-1)})$. So $(2s-1)(\ell-1)=o(n)$ because $\frac{t-1}{k-1}+\frac{1}{t}<1$ since $k \geq t+2$. Thus, by \eqref{E:proofAsymp1}, the right hand side of \eqref{E:defToNondefCondition} is
\[\mfrac{\l(2^{t+1}-1)\alpha^t}{(k-1)\cdots(k-t)}\,n+o(n)\]
and hence \eqref{E:defToNondefCondition} holds for all sufficiently large $n$ because $\alpha$ obeys \eqref{E:k>=t+2}.
\end{proof}

\section{Conclusion}

Our proof uses the fact that we are sequencing a partial $(n,k,t)_\l$-system, rather than an arbitrary $k$-uniform hypergraph, at only two points:
\begin{itemize}
    \item
to bound the maximum number of blocks that may have all but one of their points in one of two overlapping sets of consecutive vertices in Lemma~\ref{L:defToNondef},
    \item
to bound the maximum number of blocks that may intersect a given block in Lemma~\ref{L:findDef}.
\end{itemize}
This suggests that our approach here may be able to be generalised to other settings, including to partial directed designs. Good orderings of partial directed triple systems have been previously considered in \cite{KreStiVei1,KreStiVei2}.

We conclude by mentioning one more connection between $\ell$-good sequencings and other parameters of hypergraphs. In \cite{PavSanSte} the authors say that a $k$-uniform hypergraph contains the \emph{$j$th power of a hamilton cycle} if its vertices can be cyclically ordered so that every $k$-subset of $k+j-1$ consecutive vertices is one of its hyperedges. Then the existence of an $\ell$-good ordering of an $(n,k,t)$-Steiner system $(X,\B)$ is exactly equivalent to the existence of the $(\ell-k+1)$st power of a hamilton cycle in the hypergraph $(X,\overline{\B})$ where $\overline{\B}=\{E \subseteq X: |E|=k, E \notin \B\}$. This is a very different regime to the one that \cite{PavSanSte} focusses on, where the power of the hamilton cycle is constant as the order of the hypergraph grows. 

\bigskip
\noindent\textbf{Acknowledgments.}
The first author was supported by Australian Research Council grants DP150100506 and FT160100048.

\bigskip
\noindent\textbf{Data availability statement.}
Data sharing not applicable to this article as no datasets were generated or analysed during the current study.


\begin{thebibliography}{99}

    \bibitem{AloSpe}
N. Alon and J.H. Spencer, The probabilistic method, 3rd ed., John Wiley (2008).

    \bibitem{BlaEtz}
S.R. Blackburn and T. Etzion,
Block-avoiding point sequencings, {\it J. Combin. Des.} {\bf 29} (2021), 339--366.

    \bibitem{ErsGri}
G. Erskine and T. Griggs,
Good point sequencings of Steiner triple systems,
arXiv:2204.02732 (2022).

    \bibitem{EusVer}
A. Eustis and J. Verstra\"{e}te,
On the independence number of Steiner systems, {\it Combin. Probab. Comput.} {\bf 22} (2013), 241--252.

    \bibitem{KosMubRodTet}
A. Kostochka, D. Mubayi, V. R\"{o}dl and P. Tetali,
On the chromatic number of set systems, {\it Random Structures Algorithms} {\bf 19} (2001),  87--98.

    \bibitem{KosMubVer}
A. Kostochka, D. Mubayi and J. Verstra\"{e}te,
On independent sets in hypergraphs, {\it Random Structures Algorithms} {\bf 44} (2014), 224--239.

    \bibitem{KreSti}
D.L.\ Kreher and D.R.\ Stinson,
Block-avoiding sequencings of points in Steiner triple systems, {\it Australas. J. Combin.} \textbf{74} (2019), 498--509.

    \bibitem{KreStiVei1}
D.L.\ Kreher, D.R. Stinson and S. Veitch,
Block-avoiding point sequencings of directed triple systems, {\it Discrete Math.} {\bf 343} (2020), 111773, 10 pp.

    \bibitem{KreStiVei2}
D.L.\ Kreher, D.R. Stinson and S. Veitch,
Block-avoiding point sequencings of Mendelsohn triple systems, {\it Discrete Math.} {\bf 343} (2020), 111799, 7 pp.

    \bibitem{PavSanSte}
M. Pavez-Sign\'{e}, N. Sanhueza-Matamala and M. Stein,
Towards a hypergraph version of the P\'osa-Seymour conjecture,
arXiv:2110.09373 (2021).

    \bibitem{RodSiv}
V. R\"{o}dl and E. \v{S}i\v{n}ajov\'{a}, Note on independent sets in Steiner systems, {\it Random Structures Algorithms} {\bf 5} (1994), 183--190.

    \bibitem{StiVei}
D.R. Stinson and S. Veitch, Block-avoiding point sequencings of arbitrary length in Steiner triple systems, {\it Australas. J. Combin.} \textbf{77} (2020), 87--99.



\end{thebibliography}
\end{document}